\documentclass[12pt, reqno]{amsart}
\usepackage{latexsym, amsmath, amssymb, amsthm, mathscinet, mathtools, enumerate}
\usepackage{nicematrix}
\usepackage{tikz}
\usetikzlibrary{matrix, fit}
\usepackage{cases, verbatim}
\usepackage[active]{srcltx}
\usepackage{hyperref}
\usepackage{xcolor}
\usepackage{enumitem}

\usepackage[margin=1.5in]{geometry}

\usepackage{tikz}
\usetikzlibrary{arrows.meta,decorations.pathreplacing}
\usetikzlibrary{calc,decorations.pathreplacing}
\newcommand{\CurveLabel}{\Gamma_2}

\newtheorem{theorem}{Theorem}[section]
\newtheorem{definition}{Definition}[section]
\newtheorem{lemma}[theorem]{Lemma}
\newtheorem{corollary}[theorem]{Corollary}
\newtheorem{example}{Example}[section]
\newtheorem{proposition}[theorem]{Proposition}
\newtheorem{rem}{Remark}
\newtheorem*{rem*}{Remark}
\newcommand{\R}{\mathbb{R}}
\newcommand{\N}{\mathbb{N}}
\newcommand{\Z}{\mathbb{Z}}

\newcommand{\TT}{\mathbb{T}_{\theta,t}}

\newcommand{\De}{\mathbb{D}_{\varepsilon}}
\newcommand{\m}{\boldsymbol{m}}
\newcommand{\HH}{{\mathbb{H}}}
\newcommand{\PP}{{\mathbb{P}}}
\newcommand{\EE}{\mathbb{W}(\theta)}
\newcommand{\PHI}{\boldsymbol{\Phi}}
\newcommand{\en}{\boldsymbol{e_n}}
\newcommand{\WW}{\mathbb{W}}
\newcommand{\Q}{\mathcal{Q}}

\title[$L^p$ integrability of functions]{$L^p$ integrability of functions with Fourier support on High codimension manifolds}            

\author[C. Bi]{Cheng Bi}
\address[C. Bi]{Chern Institute of Mathematics, Nankai University, Tianjin 300071, P. R. China, {\tt cbi24158@gmail.com or 9820260175@nankai.edu.cn}}

\author[R. Sun]{Ruichen Sun}
\address[R. Sun]{School of Mathematics and Statistics, Northwestern Polytechnical University, Xi'an 710129, P. R. China, {\tt sunruichen667@gmail.com}}

\begin{document}

	\maketitle
	
	\begin{abstract}
     Let $n,k\ge 2$ and $N_{n,k} = \binom{n+k}{k}-1$. We prove that if $f \in L^{p}(\R^{N_{n,k}})$ and $\hat{f}$ is supported on the $n$-dimensional Parsell-Vinogradov manifold, then $f \equiv 0$ whenever $1\le p <  p_{c}(n,k) = n+ \binom{n+k}{n+1}$. This improves the bound by Agranovsky and Narayanan, and generalizes the moment-curve result of Guo-Iosevich-Zhang-Zorin-Kranich. The exponent $p_{c}(n,k)$ is sharp when $n=2$ or $k=2$, since it matches the convergence exponent $p_{n,k}^{\mathrm{Tar}}$ of Tarry's problem. Our result also yields a general lower bound $p_{n,k}^{\mathrm{Tar}} \ge p_{c}(n,k)$, improving Ikromov's bound by $n-1$. The proof combines the rotation symmetry of Parsell-Vinogradov manifold with the essential orthogonality of rotated covering boxes.
	\end{abstract}

	\section{Introduction}
    The integrability of a function can be influenced by the geometry of the support of its Fourier transform. For a given closed set $S\subset \R^N$, a basic question is to determine the range of $p$ such that if
    \begin{equation}\label{S}
        f\in L^p(\mathbb R^N) \quad \mbox{and} \quad
        \operatorname{supp}\widehat f\subset S,
        \quad \mbox{then}\quad
        f\equiv 0.
    \end{equation}

 In general, the answer depends on the dimension of $S$. Agranovsky and Narayanan
 \cite{AgranovskyNarayanan2004} proved that if $S$ is a smooth $s$-dimensional submanifold, then \eqref{S} holds provided $1\le p \le \frac{2N}{s}$. This range is essentially sharp when $s\ge \frac{N}{2}$.   Senthil Raani
 \cite[Theorem 2.3]{SenthilRaani2014} subsequently obtained related
 results for Fourier transforms supported on sets of finite packing
 measure. However, for submanifolds of high codimension that carry additional curvature and algebraic structure, the range need not be
 sharp.

 An important example is the moment curve
 \[
 \Gamma_k = \left\{  (t,t^2,\ldots,t^k) \, : \, t\in [0,1] \right\}.
 \]
  Guo, Iosevich, Zhang and Zorin-Kranich
 \cite{GuoIosevichZhangZorinKranich2025} proved that if $S = \Gamma_k$ or a sufficiently small smooth perturbation of $\Gamma_k$, then the sharp range in \eqref{S} is $1\le p < 1+ \frac{k(k+1)}{2}$.
 It also agrees with the critical output exponent
 arising in Fourier extension for the moment curve; see
 \cite{Drury1985}. More recently, Duan, Pham and Ryou
 \cite[Theorem~1.1]{DuanPhamRyou2025} established endpoint and fractal analogues for
 subsets of the moment curve.
 \subsection{Statement of the main result}\label{sec-sta}
 We establish a higher-dimensional analogue of the moment curve uniqueness theorem for Parsell-Vinogradov manifold.
Let $n\ge 1$, $k\ge 2$.  Set 
\[
N_{n,k} =   \binom{n+k}{k}-1.
\]
For $x= (x_1,\cdots, x_n) \in \R^n$ and multi-index $\alpha = (\alpha_1,\cdots,\alpha_n)\in \N^n$, we write $x^\alpha = x_1^{\alpha_1}\cdots x_n^{\alpha_n}$.
Define
\begin{equation*}
    \PHI(x) = \PHI_{n,k}(x) = \left( x^\alpha \right)_{1 \le |\alpha | \le k}.
\end{equation*}
The Parsell-Vinogradov manifold is defined as
\[
\mathcal{V}_{n,k} : = \PHI \left( [0,1]^n \right),
\]
which is an $n$-dimensional manifold in $\R^{N_{n,k}}$.
 Note that $\mathcal{V}_{1,k}$ corresponds to the case of the moment curve.

Set
\[
p_{c}(n,k) = n + \binom{n+k}{n+1}.
\]
Our main result is as follows:
\begin{theorem}\label{thm:main}
    Let $n\ge 2$ and $k\ge 2$. Suppose that $1\le p < p_c(n,k)$. If $f\in L^p (\R^{N_{n,k}})$ and $\, \mathrm{supp}\hat{f} \subset \mathcal{V}_{n,k}$, then $f\equiv 0$.
\end{theorem}

 As we will illustrate in Section \ref{sec-tar}, Theorem \ref{thm:main} is sharp up to the endpoint when $n=2$ or $k=2$. We believe that Theorem \ref{thm:main} is also sharp when $n\ge 3$ and $k\ge3$, for this reason it may be necessary to extend  \cite[Theorem 1.2]{BasuGuoZhangZorinKranich2026} to higher dimensions. Moreover, the range $1\le p < p_c(n,k)$ is better than the general result in \cite{AgranovskyNarayanan2004}, which only gives $1\le p \le \frac{2 N_{n,k}}{n}$.

 For the moment curve, the proof in \cite{GuoIosevichZhangZorinKranich2025} constructs smooth Fourier cutoffs adapted to Frenet boxes of scale $\delta \times \delta^2 \times \cdots \times \delta^k$ that covering $\Gamma_k$, and uses an interpolation argument. For $\mathcal{V}_{n,k}$ with $n\ge 2$, however, the natural covering method may only yield the range $1\le p < 1+ \binom{n+k}{n+1}$. To reach the full range in Theorem \ref{thm:main}, we use the rotaion symmetry of Parsell-Vinogradov surface.  This idea has already been used in \cite[Section 6]{BasuGuoZhangZorinKranich2026} to obtain lower bound of the convergence exponent of Tarry's problem.
The projections of these covering boxes onto the first $n$ coordinates are rectangular tiles of scale $1\times  \cdots \times 1 \times \delta$ (see the picture below for the case $n=2$).  
\begin{figure}[htbp]
\centering
\resizebox{0.7\linewidth}{!}{%
\begin{tikzpicture}[
  x=1cm, y=1cm,
  line cap=round, line join=round,
  font=\normalsize,
  axis/.style={line width=1pt,-{Latex[length=2.8mm]}},
  curve/.style={line width=.95pt},
  tube/.style={line width=.65pt,fill=none}
]

\begin{scope}
  \draw[axis] (0,0) -- (4.05,0) node[right] {$t$};
  \draw[axis] (0,0) -- (0,2.95);

  \foreach \x in {.76,1.52,2.28,3.04}
    \draw[tube] (\x,0) -- (\x,.09);
  \node at (1.95,-.28) {$\cdots$};

  \draw[
    decorate,
    decoration={brace,mirror,amplitude=4pt},
    line width=.7pt
  ]
    (0,-.20) -- (.76,-.20)
    node[midway,below=6pt] {$\delta$};

  \draw[curve,domain=0:3.50,samples=100,smooth]
    plot (\x,{.21*\x*\x});
  \node[above] at (3.50,2.70) {$\CurveLabel$};

  \draw[tube] (0,0) rectangle (.76,.18);

  \foreach \t in {1.10,3.13}{
    \pgfmathsetmacro{\yy}{.21*\t*\t}
    \pgfmathsetmacro{\ang}{atan(.42*\t)}
    \begin{scope}[shift={(\t,\yy)},rotate=\ang]
      \draw[tube] (-.39,-.085) rectangle (.39,.085);
    \end{scope}
  }
  \fill (.76,{.21*.76*.76}) circle (1.5pt);

  \foreach \t in {1.75,1.96,2.17}
    \fill (\t,{.21*\t*\t+.25}) circle (.8pt);

  \node at (3.36,1.02) {$\delta\times\delta^2$};

  \node[
    anchor=north,
    align=center,
    font=\small,
    inner sep=0pt
  ] at (2.00,-1.00)
    {Covering boxes of $\Gamma_2$};
\end{scope}

\begin{scope}[xshift=6.35cm,yshift=0.05cm]
  \def\NumStrips{5}
  \def\StripLength{3.20}
  \def\StripHeight{2.80}
  \def\RotationAngle{-45}

  \pgfmathsetmacro{\StripWidth}{\StripHeight/\NumStrips}
  \pgfmathtruncatemacro{\LastStrip}{\NumStrips-1}
  \pgfmathsetmacro{\CenterX}{\StripLength/2}
  \pgfmathsetmacro{\CenterY}{\StripHeight/2}

  \begin{scope}[
    rotate around={\RotationAngle:(\CenterX,\CenterY)}
  ]
    \foreach \j in {0,...,\LastStrip}{
      \draw[tube]
        (0,{\j*\StripWidth})
        rectangle
        (\StripLength,{(\j+1)*\StripWidth});
    }
  \end{scope}

  \foreach \j in {0,...,\LastStrip}{
    \draw[tube]
      (0,{\j*\StripWidth})
      rectangle
      (\StripLength,{(\j+1)*\StripWidth});
  }

  \draw[axis] (0,0) -- (4.10,0) node[right] {$x_1$};
  \draw[axis] (0,0) -- (0,3.20) node[above left] {$x_2$};

  \node at (3.30,3.30) {$1\times\delta$};

  \draw[axis]
    (2.00,3.62)
    .. controls (1.82,3.92) and (1.40,3.92)
    .. (1.05,3.78);

  \node[
    anchor=north,
    align=center,
    font=\small,
    inner sep=0pt
  ] at (1.95,-1.05)
    {Projections of covering boxes\\
     of $\mathcal{V}_{2,k}$};
\end{scope}
\end{tikzpicture}}
\end{figure}
We shall define Fourier multiplier by first rotating these boxes and then taking averages, see Definition \ref{def-FM}. The rotation will be realized by a shear transform. 
 
 Another main new ingredient is the observation that the rotated covering boxes are essentially orthogonal. In other words, it holds that
\begin{equation}\label{orth}
    \Big\|\sum_{\theta } {\eta}_{\theta} \Big\|_{L^2}^2\sim \sum_{\theta }\| {\eta}_{\theta} \|_{L^2}^2,
\end{equation}
where for each rotation angle $\theta$, the smooth Fourier multiplier $\eta_\theta$ is the sum of a partition of unity near $\mathcal{V}_{n,k}$, see Section \ref{2.1} for the notations. In the moment curve case, the essential orthogonality directly comes from the disjointness of translated 
boxes, whereas in higher dimensions one needs to establish the volume estimates of the intersection of boxes in different directions, see Theorem \ref{volest}.

    \subsection{Fourier extension and the  Tarry's problem}\label{sec-tar}
Theorem \ref{thm:main} has an immediate interpretation in terms of Fourier extension. Define
\[
	E_{\boldsymbol{\Phi}}g(y)
	=
	\int_{[0,1]^n}
	g(\xi)e^{iy\cdot\boldsymbol{\Phi}(\xi)}\,d\xi,
	\qquad
	y\in\mathbb R^{N_{n,k}}.
	\]
    The extension problem asks for which exponents $p,q$ an estimate of the form
   \begin{equation}\label{ex}
    \|E_{\boldsymbol{\Phi}}g\|_{L^p(\mathbb R^{N_{n,k}})}
	\leq C_{p,q}\|g\|_{L^q([0,1]^n)}
   \end{equation}
    can hold. Since $E_{\boldsymbol{\Phi}}g\ne 0$ for nonzero $g$, and  supp$\,\widehat{E_{\boldsymbol{\Phi}}g} \subset \mathcal{V}_{n,k}$, it follows from Theorem \ref{thm:main} that 
    \[
    E_{\boldsymbol{\Phi}}g \notin L^p(\R^{N_{n,k}}), \quad \mbox{provided} \quad 1\le p < p_c(n,k).
    \]
    Therefore, the global extension estimate \eqref{ex} (and its dual version) fails for such $p$.  Theorem \ref{thm:main} is indeed stronger than the failure
	of a particular extension estimate.
    The Fourier transform of an
	arbitrary $L^p$ function in the theorem may be any distribution
	supported on $\mathcal V_{n,k}$ and need not be a measure of the form
	$\widehat{E_{\PHI} g}$. Thus Theorem \ref{thm:main} treats a larger class of
	Fourier distributions than those produced by the standard extension
	operator.

    On the other hand, the sharpness of Theorem \ref{thm:main} is related to the convergence exponent of Tarry's problem, which is defined to be
    \begin{equation*}
    p^{\mathrm{Tar}}_{n,k}
	:=
	\inf\left\{
	p\geq1:
	E_{\boldsymbol{\Phi}}1\in L^p(\mathbb R^{N_{n,k}})
	\right\}.
    \end{equation*}
    The problem of determining such exponent goes back to the work of Hua
	\cite{Hua1952}. Arkhipov, Karatsuba and Chubarikov
	\cite[Theorem~5]{ArkhipovKaratsubaChubarikov1980} proved the one-dimensional
	result $p^{\mathrm{Tar}}_{1,k}=1+\frac{k(k+1)}2$, and obtained a general upper bound
	$
	p^{\mathrm{Tar}}_{n,k}
	\leq
	n\binom{n+k}{n+1}+n
	$ \cite[Theorem~7]{ArkhipovKaratsubaChubarikov1980}.
	The quadratic case was determined by Mockenhaupt
	\cite{Mockenhaupt1996}, who proved $p^{\mathrm{Tar}}_{n,2}=2(n+1)$.
	More recently, Basu, Guo, Zhang and Zorin-Kranich
	\cite[Theorem~1.2]{BasuGuoZhangZorinKranich2026} resolved the problem in dimension
	two for every $k\geq2$ with $p^{\mathrm{Tar}}_{2,k}=2+\binom{k+2}{3}$.
     Theorem \ref{thm:main} agrees with all currently known
exact cases in Tarry's problem. In particular,  \[
    p_{n,k}^{\mathrm{Tar}} = p_{c}(n,k), \quad \mbox{for} \quad n=2 \ \ \mbox{or} \ \ k=2.
    \]
    Thus, Theorem \ref{thm:main} is sharp\footnote{Recently, we have found some evidence that Theorem \ref{thm:main} is also sharp in higher dimensions.} in the case $n= 2$ or $k=2$.
    
    Moreover, it is direct from \cite[Theorem 5]{ArkhipovKaratsubaChubarikov1980} and Theorem \ref{thm:main} that 
    \begin{corollary}\label{cor-Tar}
      For all $n\ge 1$ and $k\ge 2$,  $p^{\mathrm{Tar}}_{n,k} \ge p_c(n,k)$. 
    \end{corollary}
    This improves the general lower bound $
	p^{\mathrm{Tar}}_{n,k}
	\geq
	1+\binom{n+k}{n+1}
	$ obtained by Ikromov \cite[Theorem~1]{Ikromov1997}.\\

	\section{Construction of the Fourier multiplier}\label{2.1}
    We shall construct a family of smooth cutoffs $\{\eta_{\varepsilon}\}_{\varepsilon>0}$ that is supported in a neighborhood of $\mathcal{V}_{n,k}$. If $\eta_\varepsilon \equiv 1$ near $\mathcal{V}_{n,k}$, then $f=f\ast\check{\eta}_\varepsilon$ provided supp$\, \hat{f} \subset \mathcal{V}_{n,k}$. And Theorem \ref{thm:main} will be proved via proper $L^p$ estimates of $\check{\eta}_\varepsilon$, see Section \ref{4}. 

     Unless otherwise stated, the set $\{\alpha \in \N^n : 1\le |\alpha | \le k \}$ is is order by ``$\prec$" as follows:
	\begin{enumerate}\label{order}
		\item[1.] first, by increasing the total degree \(|\alpha|=\sum_{i=1}^{n}\alpha_i\);
        \item[2.] for the same total degree, by increasing  \(\alpha_n\);
		\item[3.] when both the total degree and $\alpha_n$ agree, by decreasing lexicographical order of the remaining indices.
	\end{enumerate}

    Let $\{\boldsymbol{e_i}\}_{1\leq i\leq n}$ be a standard orthonormal basis of $\mathbb{R}^n$.
For $\theta=(\theta_1,\theta_2,\ldots,\theta_{n-1})\in \R^{n-1}$, we define the shear transform
    \begin{equation}\label{shear}
      \WW(\theta) = (\WW_1(\theta),\cdots,\WW_n(\theta)) =   \begin{pmatrix}
\mathbb{I}_{n-1} & 0 \\
\theta^T & 1 
\end{pmatrix}.
    \end{equation}
    The sheared Parsell-Vinogradov parametrization associated with 
	\(\WW(\theta)\) is 
	\begin{equation*}
		\boldsymbol{\Phi}_{\theta}(x)
		=
		\boldsymbol{\Phi}\left( \WW(\theta) x \right).
	\end{equation*}
	\begin{definition}\label{defT}
		Let $\theta \in [0,1]^{n-1}$ and  $t\in [-1,1]$.
		The Parsell-Vinogradov frame matrix corresponding to
		\((\theta,t)\) is 
		\begin{equation*}
			\mathbb{T}_{\theta,t}
			:=
			\left(\mathbb{T}_{\theta,t}^\alpha\right)_{1\le |\alpha| \le k},
		\end{equation*}
		where colume vectors $\mathbb{T}_{\theta,t}^\alpha$ 
		are determined by
        \begin{equation*}
\boldsymbol{\Phi}_{\theta}(x + t \en) - \boldsymbol{\Phi}_{\theta}(t \en) = \sum_{1\le |\alpha| \le k} \mathbb{T}_{\theta,t}^{\alpha} \, x^\alpha.
        \end{equation*}
	\end{definition}

	We shall see below that \(\TT \in \mathrm{GL}(N_{n,k})\) is a block lower triangular matrix. For $2\le j\le k$ and $|\alpha|=j$, the first $\binom{n+j-2}{n-1}-1$ entries of $\TT^\alpha$ are zero. 
    This observation will play an important role in the following proofs.
	The example for the case $n=k=2$ is provided below.
	\begin{example}[Parsell-Vinogradov Frame for $\mathcal{V}_{2,2}$]\label{example}
		\begin{equation*}
			\TT= 
			\begin{pmatrix}
1 & 0 & 0 & 0 & 0\\
\theta & 1 & 0 & 0 & 0\\
0 & 0 & 1 & 0 & 0\\
t & 0 & \theta & 1 & 0\\
2t\theta & 2t & \theta^{2} & 2\theta & 1
\end{pmatrix}
		\end{equation*}
	\end{example}

   Choose $\varphi\in C_0^\infty (\R^{N_{n,k}})$ with $0\le \varphi \le 1$, supp$\,\varphi \subset[-2,2]^{N_{n,k}}$, and $\varphi \equiv 1$ on $[-1,1]^{N_{n,k}}$. Define the scaling matrix 
   \begin{equation*}
        \mathbb{D}_\varepsilon = \mbox{diag} \{ \varepsilon^{\alpha_n}\}_{1\le |\alpha| \le k}.
    \end{equation*}
    For  $r_0 \gg 1$ and $\iota \in [-r_0/\varepsilon, r_0/\varepsilon] \cap \Z$, let
    \begin{equation*}
         \kappa_{\varepsilon, \theta,\iota} (\xi) : = \varphi \left( \De^{-1} \mathbb{T}^{-1}_{\theta, \, \varepsilon \iota/r_0 }  \left( \xi - \PHI_\theta \left( \frac{\varepsilon \iota}{r_0} \en \right)\right)\right).
    \end{equation*}
    
    We construct a partition of unity. Let 
    $$ \psi_{\varepsilon,\theta, -\lfloor \frac{r_0}{\varepsilon} \rfloor } := \kappa_{\varepsilon,\theta, -\lfloor \frac{r_0}{\varepsilon} \rfloor },$$ 
    and for $\iota > -\lfloor r_0/\varepsilon\rfloor$ in the same index range, set
    \begin{equation}\label{def-psi}
        \psi_{\varepsilon, \theta, \iota} : = \kappa_{\varepsilon,\theta, \iota} \prod_{-r_0/\varepsilon \le \iota' < \iota} \left( 1 - \kappa_{\varepsilon,\theta, \iota'} \right).
    \end{equation}
    Define
    \begin{align}\label{def-eta}
         \begin{split}
             \eta_{\varepsilon,\theta} & = \sum_{-r_0/\varepsilon \le \iota \le r_0/\varepsilon} \psi_{\varepsilon,\theta, \iota}\\
         & = 1 - \prod_{-r_0/\varepsilon \le \iota \le r_0/\varepsilon} \left( 1- \kappa_{\varepsilon, \theta, \iota}\right)\in C_0^\infty (\R^{N_{n,k}}).
         \end{split}
    \end{align}
    It follows that
     $\eta_{\varepsilon,\theta} \equiv 1$ on a small open neighborhood of $\mathcal{V}_{n,k}$.

	\begin{definition}\label{def-FM}
		Let $\varepsilon>0$.
		Define the Fourier multiplier
		\begin{equation*}
			\eta_{\varepsilon}
            := \big|\N \cap [0,\varepsilon^{-1}]\big|^{-n+1} \sum_{\boldsymbol{m} \in (\N \cap [0,\varepsilon^{-1}])^{n-1} } \eta_{\varepsilon,\varepsilon\boldsymbol{m}}.
		\end{equation*}
	\end{definition}
    \begin{rem}
        This definition exploits the rotation symmetry discussed in Section \ref{sec-sta}: we sum the partition of unity in different directions $\theta = \varepsilon \boldsymbol{m}$ and then take the average in $\boldsymbol{m}$. 
       To obtain desired $L^p$ estimate of $\check{\eta}_{\varepsilon}$, the main difficulty is to bound the volume of the intersection between different covering boxes, see Theorem \ref{volest}.
    \end{rem}

    Definition \ref{defT} extends naturally to $(\theta,t) \in \R^{n-1} \times \R$. We conclude this section with a characterization of $\TT$. First, for each $R \in$ GL$(n)$,  let $\HH_R$ be the unique matrix such that 
    $$\HH_R \boldsymbol{\Phi} (x) = \boldsymbol{\Phi}(Rx),\quad x\in \R^n.$$ 
    For $t\in \R$, write $\PP_t$ the translation in $\en$ direction: 
    $$\PP_t \boldsymbol{\Phi} (x) = \boldsymbol{\Phi}(x+t\en) - \boldsymbol{\Phi}(t\en), \quad x\in \R^n.$$ 
    Note that $\De = \HH_{D_{\varepsilon}}$, where $D_\varepsilon = \mathrm{diag}\{1,\cdots,1,\varepsilon\}\in$ GL$(n)$.
    \begin{proposition}\label{prop-HP}
    Let $R,T\in$ $\mathrm{GL}(N_{n,k})$ and $t,s\in \R$. 
        \begin{itemize}
            \item[(i)]  $\HH_R$ is block diagonal: 
            $$\HH_R =\mathrm{diag}\{ H_R^{(1)},\cdots, H_{R}^{(k)} \}.$$
            Moreover, $\HH_{\mathbb{I}_n} = \mathbb{I}_{N_{n,k}}$ and $\HH_{R} \HH_T = \HH_{R\,T}$. For $R= \WW(\theta)$, each block $H_R^{(j)}$ is lower triagular with all diagonal entries equal to $1$; in particular, $\det H_R^{(j)} = 1$.
            \item[(ii)] $\PP_t$ is block lower triangular. For $1\le j\le k$, the $j$-th diagonal block is the identity matrix of the same order as $H_R^{(j)}$. Moreover, $\PP_0 = \mathbb{I}_{N_{n,k}}$ and $\PP_t \PP_s = \PP_{t+s}$.
            \item[(iii)] $\TT = \HH_{\EE} \PP_t = \PP_t \, \HH_{\EE}$ and $\det \TT =1$.
        \end{itemize}
    \end{proposition}
    \begin{proof}
    The block structure in (i) follows directly from the the homogeneity of $\PHI$ in each degree. For the case $R= \WW(\theta)$, write $x' = (x_1,\cdots, x_{n-1})$ and $\alpha' = (\alpha_1,\cdots, \alpha_{n-1})$. Then
    \[
    \left( \WW(\theta) x \right)^\alpha = \left( x'\right)^{\alpha'} \left( x_n + \theta \cdot x' \right)^{\alpha_n}.
    \]
    Each monomial $x^\beta$ in this expansion has $|\beta| = |\alpha|$ and $\beta_n \le \alpha_n$. Equality occurs only when $\beta_n = \alpha_n$, and that term has coefficient $1$.  
    Thus, under the order ``$\prec$", each $H_R^{(j)} \in $ $\mathrm{GL}\binom{n+j-1}{j}$ is lower triangular with all diagonal elements equal to $1$, so $\det H_R^{(j)}=1$. The composition law follows from
    \[
    \HH_{R} \HH_T \PHI (x) = \HH_R \PHI(Tx) = \PHI(RT x) = \HH_{R\,T} \PHI (x), \quad  x\in \R^n. 
    \]

    The proof of (ii) is similar to that of (i). 
    Now we prove (iii). The decomposition $\TT = \HH_{\EE} \PP_t$ follows from Definition \ref{defT}. Since $\EE \en = \en$, we know
    \begin{align*}
         \HH_{\EE} \PP_t \PHI(x) &= \HH_{\EE} \big( \PHI(x+t\en) -\PHI(t\en) \big)\\
         &= \PHI\big( \EE x + t\EE \en \big) -\PHI\big( t\EE \en \big) = \PP_t \, \HH_{\EE} \PHI(x).
    \end{align*}
    This proves the commutativity.
    \end{proof}

\section{Volume estimates}\label{sec-vol}
   We prove an estimate (Theorem \ref{volest} below) for the volume of intersection of the covering boxes. Such estimate yields almost orthogonality of the form \eqref{orth}, which are essential for the $L^2$ estimate of $\eta_\varepsilon$, see \eqref{apply}. The shear transform makes this property intuitive, as illustrated by the figure in Section \ref{secint}.

   Let
   \[
   \mathcal{Q} = \left(-\tfrac{1}{2}, \tfrac{1}{2} \right)^{N_{n,k}} \subset \R^{N_{n,k}}
   \]
   be the unit cube.

	\begin{definition}[Covering Box]
		Let $\varepsilon>0$, $\m=(m_1,\ldots,m_{n-1})\in \mathbb{N}^{n-1}$, $t\in \R$ and $\varepsilon\m\in[0,1]^{n-1}$. 
        The covering box $\mathcal{B}_{\varepsilon,\boldsymbol{m},t}$ is $ \mathrm{supp}(\kappa_{\varepsilon,\varepsilon\boldsymbol{m},t})$, namely,
		\begin{equation*}
			\mathcal{B}_{\varepsilon,\m,t} = \left\{ \xi \in \mathbb{R}^{N_{n,k}} : \De^{-1} \mathbb{T}_{\varepsilon\boldsymbol{m},t}^{-1} (\xi - \boldsymbol{\Phi}_{\varepsilon\boldsymbol{m}}(t \en))\in \mathrm{supp}\,\varphi \right\}.
		\end{equation*}
        	\end{definition}
            By the property of supp$\,\varphi$, there are two constants $c_1,c_2>0$ such that
		\begin{align}\label{app}
        \begin{split}
            \boldsymbol{\Phi}_{\varepsilon\m}(t \en)+ c_1 \mathbb{T}_{\varepsilon \m ,t} \De \mathcal{Q} &\subset   \mathcal{B}_{\varepsilon,\m,t}\\
            &\subset  \boldsymbol{\Phi}_{\varepsilon\m}(t \en)+ c_2 \mathbb{T}_{\varepsilon \m ,t} \De \mathcal{Q}.
        \end{split}
		\end{align}
    Set
    \begin{equation*}
         A_{n,k} := \binom{n+k}{n+1}.
    \end{equation*}
    In what follows, for two nonnegative functions $f$ and $g$, we write $f \lesssim_{n,k} g$ if there is a constant $c_{n,k}$ only depending on $n$ and $k$ such that $f \le c_{n,k} g$. And $f \sim_{n,k} g$ means $f \lesssim_{n,k} g$ and $g \lesssim_{n,k} f$. We use $|F|$ to denote the volume of a measurable set $F$.
	\begin{theorem}\label{volest}
		Let $\varepsilon>0$, $n\ge 2$, $k\ge 2$, and $\boldsymbol{m},\boldsymbol{m}'\in\mathbb{N}^{n-1}$ satisfy $\varepsilon\boldsymbol{m},\varepsilon\boldsymbol{m}'\in [0,1]^{n-1}$, and $t,t'\in \R$. 
        Then
		\begin{equation}\label{B}
\left|
\mathcal{B}_{\varepsilon,\boldsymbol{m}',t'}
\cap 
\mathcal{B}_{\varepsilon,\m,t} 
\right|
\lesssim_{n,k}
\frac{ \varepsilon^{A_{n,k}} }{ \big( 1 + |\boldsymbol{m}'-\m | \big)^{ n+1 } }.
		\end{equation}
	\end{theorem}\begin{rem}
   The exponent $n+1$ in \eqref{B} is sufficient for our proof of Theorem \ref{thm:main} (we will see from Section \ref{4} that any exponent larger than $n$ is enough), we do not pursue its sharpness here. A possible optimal exponent is
   $$B_{n,k} = \sum_{j=0}^{\lfloor \frac{k-1}{2} \rfloor} \binom{n+k-1-2j}{n}.$$
\end{rem}

\subsection{Proof of Theorem \ref{volest}}

By \eqref{app}, it suffices to prove
    \[
    \left| \big( a+ \mathbb{T}_{\varepsilon \m ,t} \De \mathcal{Q} \big) \cap \big( b+ \mathbb{T}_{\varepsilon \m' ,t'} \De \mathcal{Q}\big) \right| \lesssim_{n,k}
\frac{ \varepsilon^{A_{n,k}} }{ \left( 1 + |\boldsymbol{m}'-\m | \right)^{ n+1 } }, 
    \]
 uniformly in $ a,b\in \R^{N_{n,k}}.$

Since $\TT \De \mathcal{Q}$ is a symmetry convex body, the translation does not cause growth of the volume. In other words, 
 \[
     \left| \big( a+ \mathbb{T}_{\varepsilon \m ,t} \De \mathcal{Q} \big) \cap \big( b+ \mathbb{T}_{\varepsilon \m' ,t'} \De \mathcal{Q}\big) \right|
     \le
     \left|  \mathbb{T}_{\varepsilon \m ,t} \De \mathcal{Q}  \cap  \mathbb{T}_{\varepsilon \m' ,t'} \De \mathcal{Q} \right|.
     \]
 It can be proved by the well-known Brunn-Minkowski inequality, see for instance \cite[pp. 376-377]{Gar02}. 
 
    Consequently, it remains to show that
    \begin{equation}\label{TDQ}
         \left|  \mathbb{T}_{\varepsilon \m ,t} \De \mathcal{Q}  \cap  \mathbb{T}_{\varepsilon \m' ,t'} \De \mathcal{Q} \right|
         \lesssim_{n,k}
         \frac{ \varepsilon^{A_{n,k}} }{ \left( 1 + |\boldsymbol{m}'-\m | \right)^{ n+1 } }.
    \end{equation}
    Since $\det \left(\mathbb{T}_{\varepsilon \m,t} \De \right) = \varepsilon^{A_{n,k}}$, a change of variable gives
    \begin{align*}
         \left|  \mathbb{T}_{\varepsilon \m ,t} \De \mathcal{Q}  \cap  \mathbb{T}_{\varepsilon \m' ,t'} \De \mathcal{Q} \right| = \varepsilon^{A_{n,k}} \left|   \mathcal{Q}  \cap  \mathbb{A} \mathcal{Q} \right|,
    \end{align*}
    where
    \begin{equation*}
        \mathbb{A}  := \De^{-1} \mathbb{T}_{\varepsilon \m,t}^{-1} \mathbb{T}_{\varepsilon \m',t'} \De.
    \end{equation*}
    Recall that
    \[
    \De = \HH_{D_\varepsilon} \quad \mbox{and} \quad D_{\varepsilon^{-1}} \EE D_{\varepsilon} = \WW(\varepsilon^{-1} \theta).
    \]
    Proposition \ref{prop-HP} gives
    \begin{align*}
        \mathbb{A} = \HH_{D_{\varepsilon^{-1}}} \HH_{\WW(\varepsilon (\m'-\m))} \PP_{t'-t} \HH_{D_\varepsilon} &= \HH_{\WW(\m'-\m)} \left( \HH_{D_{\varepsilon^{-1}}} \PP_{t'-t} \HH_{D_\varepsilon} \right)\\
        & = \HH_{\WW(\m'-\m)} \PP_{\frac{t'-t}{\varepsilon}} = \mathbb{T}_{\m'-\m, \frac{t'-t}{\varepsilon}}.
    \end{align*}
 Theorem \ref{volest} is thus reduced to the following proposition.
    \begin{proposition}\label{lem-T}
        It holds uniformly in $\m \in \Z^{n-1}$ and $t \in \R$ that 
    \begin{equation}\label{Tmt}
        |\mathcal{Q} \cap \mathbb{T}_{\m, t} \, \mathcal{Q}| \lesssim_{n,k} \left( 1 + |\m | \right)^{ -n-1 }.
    \end{equation}
      \end{proposition}
      \begin{proof}
        The case $\m = 0$ is easy. Now we assume $\m \ne 0$.  Choose a rotation $\mathbb{M} \in O(n-1)$ such that $\frac{\m}{|\m|} = \mathbb{M} \boldsymbol{e_1}$. By Proposition \ref{prop-HP}, we have
          \begin{equation*}
              \mathbb{T}_{\m,t} = \HH_{\widetilde{\mathbb{M}}} \, \mathbb{T}_{|\m| \boldsymbol{e_1},t} \HH_{\widetilde{\mathbb{M}}^T}, \quad \mbox{where} \ \  \widetilde{\mathbb{M}} = 
              \begin{pmatrix}
                \mathbb{M} & 0 \\
                0 & 1 
              \end{pmatrix}.
          \end{equation*}
          Note that $c_1 \mathcal{Q} \subset \HH_{\widetilde{\mathbb{M}}^T} \mathcal{Q} \subset  c_2 \mathcal{Q}$. Thus, it suffices to prove \eqref{Tmt} in the case $\m = |\m| \boldsymbol{e_1}$. 
          
          We claim that $\mathbb{T}_{|\m| \boldsymbol{e_1},t} $ has a minor $\mathbb{J}$ with $|\det \mathbb{J}| = |\m|^{n+1}.$ Under the order ``$\prec$" (see Section \ref{2.1}), we can identify the column and row indices of $\TT$ with $\{x^\alpha : 1\le |\alpha| \le k  \}$. Choose
          \begin{align*}
          \begin{split}
                \mathcal{I} &= (x_n, x_2 x_n, \cdots, x_{n-1}x_n, x_n^2),\\
                \mathcal{J} &= (x_1, x_1^2, x_1 x_2, \cdots, x_1 x_{n-1}).
          \end{split}
          \end{align*}
          Let $\mathbb{J}$ be the minor with rows indexed by $\mathcal{I}$ and columns indexed by $ \mathcal{J}$. A direct computation gives
          \[
          \mathbb{J} = |\m|\begin{pmatrix}
                1 & 0 & 0 \\
                0 & 0  &  \mathbb{I}_{n-2} \\
               2t  & |\m|  & 0
              \end{pmatrix}.
          \]
          It is clear that $|\det \mathbb{J}| = |\m|^{n+1}$. 
A change of variables gives
\[
  |\Q\cap \mathbb T_{|\m|\boldsymbol e_1,t}\Q|
  =\int_\Q\mathbf1_\Q(\mathbb T_{|\m|\boldsymbol e_1,t} \, x)\,dx
  =|\{x\in\Q: \mathbb T_{|\m|\boldsymbol e_1,t} \, x\in\Q\}|.
\]
Let $\mathcal K$ be the complement of $\mathcal J$ in the
index set, write
$x=(x_{\mathcal J},x_{\mathcal K})$ and
$\Q=\Q_{\mathcal J}\times\Q_{\mathcal K}$.
There exists A minor $\mathbb{K}$ such that the rows indexed by $\mathcal I$ are
\[
  (\mathbb T_{|\m|\boldsymbol e_1,t} \, x)_{\mathcal I}
  =\mathbb J x_{\mathcal J} + \mathbb{K} x_{\mathcal K}.
\]
For each fixed $x_{\mathcal K}\in\Q_{\mathcal K}$, the fiber
\begin{align*}
     W_{x_{\mathcal{K}}} &:= \left\{ x_{\mathcal{J}}:(x_\mathcal{J},x_{\mathcal{K}}) \in \mathcal{Q} \cap \mathbb{T}_{|\m| \boldsymbol{e_1},t} \mathcal{Q} \right\}  \subset  \left\{ x_{\mathcal{J}}: \mathbb{B} x_{\mathcal{J}} + \mathbb{K} x_{\mathcal{K}} \in \mathcal{Q}_{\mathcal{J}}\right\} =: F_{x_{\mathcal{K}}}.
\end{align*}
Since $|\Q_{\mathcal I}|=1$ and $|F_{x_{\mathcal K}}|
  =|\det\mathbb J|^{-1}=|\m|^{-n-1}$.
Fubini's theorem now gives
\begin{align*}
  |\Q\cap \mathbb T_{|\m|\boldsymbol e_1,t} \Q|
  =\int_{\Q_{\mathcal K}}|W_{x_{\mathcal K}}|\,dx_{\mathcal K}\leq\int_{\Q_{\mathcal K}}|F_{x_{\mathcal K}}|\,dx_{\mathcal K}
   =|\m|^{-n-1}.
\end{align*}
Combining this with $|\Q\cap \mathbb T_{|\m|\boldsymbol e_1,t} \Q|\leq1$ proves \eqref{Tmt}.
\end{proof}

\begin{rem}
There is another way to deal with the volume estimates. 
Taylor expansion gives
$$\mathbb{T}^{\boldsymbol{\alpha}}_{
  \varepsilon\mathbf{m}
  +(\varepsilon\mathbf{m}'-\varepsilon\mathbf{m}),\,t+(t'-t)}
=
\sum_{1\leq|\boldsymbol{\beta}|\leq k}
a_{\boldsymbol{\beta}}
\bigl( \varepsilon (\mathbf{m}'-\mathbf{m}), \,t'-t\bigr)
\mathbb{T}^{\boldsymbol{\beta}}_{\varepsilon\mathbf{m},t}.$$
We can therefore get $\mathbb{T}_{\varepsilon \m,t}^{-1} \mathbb{T}_{\varepsilon \m',t'}$ by computing $a_{\boldsymbol{\beta}}$. These computation can be simplified using tensor differentiation, where we use the observation that from $(n+1)$-th to $(n+\frac{n(n+1)}{2})$-th coordinates of $\mathbb{T}_{\varepsilon\boldsymbol{m},t}^{\boldsymbol{e}_i+\boldsymbol{e}_j}$ are the coordinates of $\WW_i\otimes\WW_j+\WW_j\otimes\WW_i$ (and of $\WW_j\otimes\WW_j$, if $i=j$) with respect to the natural basis of $\mathrm{Sym}^2(\mathbb{R}^n)$. 
\end{rem}

\subsection{The amount of intersecting boxes}\label{secint}
    Except the volume, we need to estimate the intersection number 
\begin{equation*}
    \left|\{\iota'\in\mathbb{Z}\ |\ \mathcal{B}_{\varepsilon,\m',\varepsilon\iota'}\cap\mathcal{B}_{\varepsilon,\m,\varepsilon\iota} \ne \emptyset \}\right|, \quad \forall \, \iota\in\mathbb{Z},
\end{equation*}
which will also be used in the $L^2$ estimate of $\eta_{\varepsilon}$, see \eqref{apply}.

\begin{lemma}\label{iota}
    Let $\iota\in\mathbb{Z}$ and $\boldsymbol{m},\boldsymbol{m}'\in\mathbb{N}^{n-1}$, we have 
\begin{equation*}\label{intersec}
         \left|\{\iota'\in\mathbb{Z}\ |\ \mathcal{B}_{\varepsilon,\m',\varepsilon\iota'}\cap\mathcal{B}_{\varepsilon,\m,\varepsilon\iota} \ne \emptyset \} \right|
         \lesssim_{n,k} \, 1+ |\boldsymbol{m}-\boldsymbol{m}'|.
    \end{equation*}
\end{lemma}
\begin{rem}
    This lemma can be understood intuitively. It describes, when the angle between two families of covering boxes is $\varepsilon (\boldsymbol{m'}-\boldsymbol{m})$, what is the maximum number of boxes from one family that can intersect a given box from another family. We can also see this by considering the volume of the overlapping part, as illustrated in the figure below.
\end{rem}
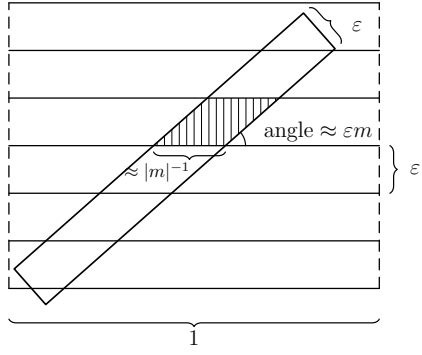
\begin{figure}[htbp]
\centering
\scalebox{0.7}{
\begin{tikzpicture}[
  x=1cm,y=1cm,
  line cap=round,
  line join=round,
  font=\normalsize,
  strip/.style={
    line width=.7pt
  },
  diagonal/.style={
    line width=.9pt
  },
  dim/.style={
    line width=.65pt,
    decorate,
    decoration={brace,amplitude=4pt}
  }
]

\def\NumStrips{6}        
\def\StripLength{7.00}   
\def\StripWidth{0.90}   
\def\TubeLength{7.30}   
\def\TubeAngle{41.5}     
\def\StartX{0.10}
\def\StartY{0.37}
\def\ShadeIndex{3}       

\pgfmathsetmacro{\TotalHeight}{
  \NumStrips*\StripWidth
}
\pgfmathtruncatemacro{\NumHatches}{
  ceil(\StripLength/.14)
}

\coordinate (A) at (\StartX,\StartY);

\coordinate (B) at
  ($(A)+(\TubeAngle:\TubeLength)$);

\coordinate (C) at
  ($(B)+({\TubeAngle-90}:\StripWidth)$);

\coordinate (D) at
  ($(A)+({\TubeAngle-90}:\StripWidth)$);

\pgfmathsetmacro{\Ylow}{
  \ShadeIndex*\StripWidth
}
\pgfmathsetmacro{\Yhigh}{
  \Ylow+\StripWidth
}
\pgfmathsetmacro{\Xleft}{
  \StartX+(\Ylow-\StartY)/tan(\TubeAngle)
}
\pgfmathsetmacro{\Xright}{
  \Xleft+\StripWidth/sin(\TubeAngle)
}
\pgfmathsetmacro{\Xshift}{
  \StripWidth/tan(\TubeAngle)
}

\coordinate (P) at (\Xleft,\Ylow);
\coordinate (Q) at (\Xright,\Ylow);
\coordinate (R) at ({\Xright+\Xshift},\Yhigh);
\coordinate (S) at ({\Xleft+\Xshift},\Yhigh);

\begin{scope}
  \clip (P)--(Q)--(R)--(S)--cycle;

  \foreach \j in {0,...,\NumHatches}{
    \draw[line width=.35pt]
      ({.14*\j},\Ylow)
      --
      ({.14*\j},\Yhigh);
  }
\end{scope}

\foreach \j in {0,...,\NumStrips}{
  \draw[strip]
    (0,{\j*\StripWidth})
    --
    (\StripLength,{\j*\StripWidth});
}

\draw[
  strip,
  dash pattern=on 6pt off 4pt
]
  (0,0)--(0,\TotalHeight);

\draw[
  strip,
  dash pattern=on 6pt off 4pt
]
  (\StripLength,0)--(\StripLength,\TotalHeight);

\draw[diagonal]
  (A)--(B)--(C)--(D)--cycle;

\draw[dim]
  ({\StripLength+.20},\Ylow)
  --
  ({\StripLength+.20},{\Ylow-\StripWidth});

\node[anchor=west] at
  ({\StripLength+.44},{\Ylow-.5*\StripWidth})
  {$\varepsilon$};

\draw[dim]
  ($(B)+(\TubeAngle:.15)$)
  --
  ($(C)+(\TubeAngle:.15)$);

\node[anchor=west] at
  ($(B)!.5!(C)+(.48,.08)$)
  {$\varepsilon$};

\draw[
  dim,
  decoration={brace,mirror,amplitude=4pt}
]
  ($(P)+(0,-.12)$)
  --
  ($(Q)+(0,-.12)$);

\node[font=\footnotesize] at
  ({(\Xleft+\Xright)/2-.61},{\Ylow-.50})
  {$\approx |m|^{-1}$};

\draw[line width=.65pt]
  ($(Q)+(.38,0)$)
  arc[
    start angle=0,
    end angle=\TubeAngle,
    radius=.38cm
  ];

\node[anchor=west] at
  ($(Q)+(.58,.27)$)
  {$\text{angle}\approx\varepsilon m$};

\draw[
  dim,
  decoration={brace,mirror,amplitude=4pt}
]
  (0,-.58)--(\StripLength,-.58);

\node at ({\StripLength/2},-.94)
  {$1$};

\end{tikzpicture}
}

\par\vspace{4pt} 

{\centering\normalsize
The shaded area is approximately $\varepsilon |m|^{-1}$, and 
the\\ number of intersection boxes is $O(|\m|)$.\par
}
\end{figure}
\begin{proof}[Proof of Lemma \ref{iota}]
Let $\mathbb{P} = (\mathbb{I}_n \ 0)$ be the projection from $\R^{N_{n,k}}$ to its first $n$ coordinates, then it is direct that
\begin{align}\label{PP}
\begin{split}
       &\left|\{\iota'\in\mathbb{Z} :  \mathcal{B}_{\varepsilon,\m',\varepsilon\iota'}\cap\mathcal{B}_{\varepsilon,\m,\varepsilon\iota} \ne \emptyset
    \} \right| \\ 
    & \qquad \le \left|\{\iota'\in\mathbb{Z}: \mathbb{P}\mathcal{B}_{\varepsilon,\m',\varepsilon\iota'}\cap \mathbb{P}\mathcal{B}_{\varepsilon,\m, \varepsilon\iota} \ne \emptyset
    \} \right|.
\end{split}
\end{align}
By \eqref{app}, it suffices to prove the desired upper bound with $\mathcal{B}_{\varepsilon,\m,\iota\varepsilon}$
and
$\mathcal{B}_{\varepsilon,\m',\iota'\varepsilon}$ in the right side of \eqref{PP} replaced by 
\[
 \PHI_{\varepsilon \boldsymbol{m}} (\varepsilon \iota \en) + \mathbb{T}_{\varepsilon \boldsymbol{m}, \varepsilon \iota } \De \mathcal{Q}  \quad \mbox{and} \quad \PHI_{\varepsilon \boldsymbol{m'}} (\varepsilon \iota' \en) + \mathbb{T}_{\varepsilon \boldsymbol{m'}, \varepsilon \iota' } \De \mathcal{Q},
\]
respectively.
 Set $\mathcal{Q}_n = (-\frac{1}{2},\frac{1}{2})^n$. By the definition of $\PHI_\theta$ and Proposition \ref{prop-HP},
\[
\mathbb{P} \big( \PHI_{\varepsilon \boldsymbol{m}} (\varepsilon \iota \en) + \mathbb{T}_{\varepsilon \boldsymbol{m}, \varepsilon \iota } \De \mathcal{Q} \big) = \varepsilon \iota \en + \WW(\varepsilon \boldsymbol{m}) D_\varepsilon \mathcal{Q}_n,
\]
where we recall that $D_\varepsilon = \mbox{diag}\{ 1,\cdots, 1,\varepsilon\}$.

Therefore, if such two boxes intersects, there exist $y,z \in \mathcal{Q}_n$ such that
\[
\varepsilon \iota \en + \WW(\varepsilon \boldsymbol{m}) D_\varepsilon y = \varepsilon \iota' \en + \WW(\varepsilon \boldsymbol{m'}) D_\varepsilon z .
\]
Equivalently,
\begin{equation}\label{en}
    (\iota - \iota')\varepsilon \en = \WW(\varepsilon \boldsymbol{m'}) D_\varepsilon z - \WW(\varepsilon \boldsymbol{m}) D_\varepsilon y.
\end{equation}

Writing $y=(y',y_n)$ and $z = (z' , z_n)$, we have
\[
\WW(\varepsilon \boldsymbol{m'}) D_\varepsilon z - \WW(\varepsilon \boldsymbol{m}) D_\varepsilon y   = \begin{pmatrix}
    z'-y' \\
    \varepsilon \boldsymbol{m'} \cdot z' + \varepsilon z_n - \varepsilon \boldsymbol{m} \cdot y' -  \varepsilon y_n
\end{pmatrix}.
\]
The identity \eqref{en} gives $z' = y'$, and thus
\[
\iota - \iota'  =  (\boldsymbol{m'}-\boldsymbol{m})\cdot y' +  (z_n-y_n),
\]
which yields $|\iota - \iota'| \lesssim_{n,k} |\boldsymbol{m'}-\boldsymbol{m}|+1$.
\end{proof}

	\section{Proof of Theorem \ref{thm:main}}\label{4}
    In the case $1\le p \le 2$,  Hausdorff-Young inequality implies $\hat{f} \in L^{p'}$.   Since supp$\, \hat{f} \subset \mathcal{V}_{n,k}$, we know $f  =  \hat{f} = 0$.

    Now we consider the case $2< p<p_c(n,k)$.
	Let $\eta_\varepsilon$ be as in Definition \ref{def-FM}. Then $\eta_\varepsilon \equiv 1$ near $\mathcal{V}_{n,k}$, and thus
	\begin{equation*}
f=f\ast\check{\eta}_\varepsilon.
	\end{equation*}
	 We claim that 
	\begin{equation}\label{eta}
\| \check{\eta}_\varepsilon \|_{L^{p'}}\rightarrow 0, \quad \mbox{as}\quad   \varepsilon\rightarrow 0^{+}
	\end{equation}
    whenever $2<p<p_c(n,k)$, here $\frac{1}{p}+ \frac{1}{p'}=1$. 
	This implies $\|f\|_{L^\infty}=0$ by Young's inequality, and completes the proof.

The proof of \eqref{eta} is via an $L^1$-$L^2$ interpolation argument.  By Plancherel's theorem and \eqref{def-eta}, we have 
\begin{align*}
\|
\check{\eta}_\varepsilon\|_{L^2}^2&
=
\Big\| \varepsilon^{n-1} \sum_{\boldsymbol{m}
\in\left[0,\varepsilon^{-1}\right]^{n-1} \cap \N^{n-1} } \check{\eta}_{\varepsilon,\varepsilon\boldsymbol{m}} \Big\|_{L^2}^2
=
\Big\|
\varepsilon^{n-1} \sum_{\boldsymbol{m}
} \eta_{\varepsilon,\varepsilon\boldsymbol{m}}
\Big\|_{L^2}^2 
\\
& = \varepsilon^{2n-2}
\int_{\R^{N_{n,k}}}
\sum_{\boldsymbol{m},\boldsymbol{m}' } \eta_{\varepsilon,\varepsilon\boldsymbol{m}}(\xi) \, \eta_{\varepsilon,\varepsilon\boldsymbol{m}'}(\xi) 
\,d\xi 
\\
& \le \varepsilon^{2n-2} \sum_{\boldsymbol{m},\boldsymbol{m}' } \sum_{\iota,\iota'\in[-r_0/\varepsilon,r_0/\varepsilon]} |\mathcal{B}_{\varepsilon,\boldsymbol{m},\varepsilon\iota}\cap \mathcal{B}_{\varepsilon,\boldsymbol{m}',\varepsilon\iota'}|.
		\end{align*}
    Using  Theorem \ref{volest}, we obtain
\begin{align}\label{apply}
        \begin{split}
            \sum_{\iota,\iota'\in[-r_0/\varepsilon,r_0/\varepsilon]} |\mathcal{B}_{\varepsilon,\boldsymbol{m},\varepsilon\iota}\cap  \mathcal{B}_{\varepsilon,\boldsymbol{m}',\varepsilon\iota'}| 
         &= \sum_{ \iota \in[-r_0/\varepsilon,r_0/\varepsilon]} \sum_{ \substack{ \iota'\in[-r_0/\varepsilon,r_0/\varepsilon]\\ \mathcal{B}_{\varepsilon,\boldsymbol{m}, \varepsilon \iota} \cap \mathcal{B}_{\varepsilon, \boldsymbol{m}', \varepsilon \iota'} \ne \emptyset}}  |\mathcal{B}_{\varepsilon,\boldsymbol{m},\varepsilon\iota}\cap \mathcal{B}_{\varepsilon,\boldsymbol{m}',\varepsilon\iota'}| \\
        & \lesssim_{n,k} \sum_{ \iota \in [-r_0/\varepsilon,r_0/\varepsilon]} \sum_{ \substack{  \iota' \in [-r_0/\varepsilon,r_0/\varepsilon] \\ \mathcal{B}_{\varepsilon,\boldsymbol{m}, \varepsilon \iota} \cap \mathcal{B}_{\varepsilon, \boldsymbol{m}', \varepsilon \iota'} \ne \emptyset}}  \frac{ \varepsilon^{A_{n,k}} }{ \left( 1 + |\boldsymbol{m}'-\m | \right)^{ n+1 } }.
        \end{split}
    \end{align}
 By Lemma \ref{iota}, the right side can be bounded by
    \[
      \sum_{ \iota \in[-r_0/\varepsilon,r_0/\varepsilon]}  \frac{ \varepsilon^{A_{n,k}} }{ \left( 1 + |\boldsymbol{m}'-\m | \right)^{ n } } \sim_{n,k} \frac{ \varepsilon^{ A_{n,k}-1 } }{ \left( 1 + |\boldsymbol{m}'-\m | \right)^{ n } } =: \mathcal{C}_{\boldsymbol{m'},\boldsymbol{m},\varepsilon}.
    \]
Therefore, 
\begin{align*}
\|\check{\eta}_\varepsilon\|_{L^2}^2 \lesssim_{n,k} \varepsilon^{2n-2} \sum_{\boldsymbol{m},\boldsymbol{m}' } 
\mathcal{C}_{\boldsymbol{m'},\boldsymbol{m},\varepsilon}.
\end{align*}
By dyadic decomposition in $|\boldsymbol{m'}-\boldsymbol{m}|$, it holds that
\[
\sum_{\boldsymbol{m},\boldsymbol{m}' } \mathcal{C}_{\boldsymbol{m'},\boldsymbol{m},\varepsilon}
\le
\sum_{\boldsymbol{m'}
= \boldsymbol{m}} \mathcal{C}_{\boldsymbol{m'},\boldsymbol{m},\varepsilon} + \sum_{  \boldsymbol{m} } \sum_{\ell \ge 0} \sum_{\substack{\boldsymbol{m}'\\
    2^{\ell} \le |\boldsymbol{m}'-\boldsymbol{m}| \le 2^{\ell+1} }} \mathcal{C}_{\boldsymbol{m'},\boldsymbol{m},\varepsilon}.
\]
Denote the right side by $I_1 + I_2$. It is direct that $I_1 \sim_{n,k}  \varepsilon^{A_{n,k}-n}$. Moreover,
\begin{align*}
    I_2 &\lesssim_{n,k}   \sum_{  \boldsymbol{m} } \sum_{\ell\ge 0} 2^{ (n-1)\ell } \, \frac{\varepsilon^{A_{n,k}-1}}{2^{n\ell}}\\
    &\lesssim_{n,k}  \varepsilon^{1-n} \varepsilon^{A_{n,k}-1} \sum_{\ell \ge 0} 2^{-\ell} \, \lesssim_{n,k} \,  \varepsilon^{A_{n,k}-n}.
\end{align*}
As a consequence, $\|\check{\eta}_\varepsilon\|_{L^2}^2 \lesssim_{n,k} \varepsilon^{A_{n,k}+n-2}$. Indeed, after minor modifications of the above argument, we can deduce essential orthogonality $\|\check{\eta}_\varepsilon\|_{L^2}^2 \sim_{n,k} \varepsilon^{A_{n,k}+n-2}$, for our proof the upper bound is already sufficient.

On the other hand, $\check{\eta}_\varepsilon$ has the following $L^1$ estimate.
    \begin{align*}
        \|\check{\eta}_\varepsilon\|_{L^1} 
        & =
        \varepsilon^{n-1}
        \bigg\| \sum_{\boldsymbol{m} }
        \check{\eta}_{\varepsilon,\varepsilon\boldsymbol{m}} \bigg\|_{L^1}  \le  \varepsilon^{n-1}
\sum_{\boldsymbol{m} }  \left\|\check{\eta}_{\varepsilon,\varepsilon\boldsymbol{m}} \right\|_{L^1}
        \lesssim_{n,k}
        \varepsilon^{n-1}
        \cdot\frac{1}{\varepsilon^n} =\varepsilon^{-1}.
    \end{align*}
    To prove the last ``$\lesssim_{n,k}$" above,  by \eqref{def-eta} and \eqref{def-psi},
    \[
    \| \check{\eta}_{\varepsilon,\varepsilon\boldsymbol{m}} \|_{L^1} \le \sum_{\iota\in [-r_0/\varepsilon,r_0/\varepsilon]}  \| \check{\psi}_{\varepsilon,\varepsilon \m,\iota} \|_{L^1} \lesssim_{n,k} \sum_{\iota\in [-r_0/\varepsilon,r_0/\varepsilon]} 1 \le 2r_0 \, \varepsilon^{-1}, 
    \]
    where we have used the facts   that the supports of $\{\kappa_{\varepsilon,\theta,\iota} \}_\iota$ are finitely overlapping (by Lemma \ref{iota} in the case $\m = \m'$) and $\| \check{\kappa}_{\varepsilon,\theta, \iota} \|_{L^1} \lesssim_{n,k} 1$.
    
	Interpolation between $L^1$ and $L^2$ gives 
    \[
    \| \check{\eta}_\varepsilon \|_{L^{p'}} \lesssim_{n,k} \varepsilon^{p_c(n,k) -1 - \frac{p_c(n,k)}{p'}}=\varepsilon^{\frac{p_c(n,k)}{p}-1},\quad p_c(n,k) = n + \binom{n+k}{n+1},
    \]
    which yields \eqref{eta}. The proof is completed.

    \section*{Acknowledgments}
     We would like to thank Professors Shaoming Guo, Ruixiang Zhang and Pavel Zorin-Kranich for helpful discussions, especially for introducing us the method of rotation symmetry. We also thank Professor Pavel Zorin-Kranich for pointing out a simplification in the original proof of Proposition \ref{lem-T}.

     \section*{AI Declaration}
In the proof of Proposition \ref{lem-T}, ChatGPT-5.6 was used to assist with elementary matrix calculations and checked by the authors, who take full responsibility for them. During the writing, ChatGPT-6 was used to generate the figures. 

	\bibliographystyle{amsplain}
	\bibliography{references2}

@article {Gar02,
    AUTHOR = {Gardner, R. J.},
     TITLE = {The {B}runn-{M}inkowski inequality},
   JOURNAL = {Bull. Amer. Math. Soc. (N.S.)},
  FJOURNAL = {American Mathematical Society. Bulletin. New Series},
    VOLUME = {39},
      YEAR = {2002},
    NUMBER = {3},
     PAGES = {355--405},
      ISSN = {0273-0979,1088-9485},
   MRCLASS = {26D15 (52A40)},
MRREVIEWER = {Apostolos\ A.\ Giannopoulos},
       DOI = {10.1090/S0273-0979-02-00941-2},
       URL = {https://doi.org/10.1090/S0273-0979-02-00941-2},
}

@article{AgranovskyNarayanan2004,
	author  = {Agranovsky, M. L. and Narayanan, E. K.},
	title   = {{$L^p$-integrability, supports of {Fourier} transforms
			and uniqueness for convolution equations}},
	journal = {J. Fourier Anal. Appl.},
	volume  = {10},
	number  = {3},
	pages   = {315--324},
	year    = {2004},
	doi     = {10.1007/s00041-004-0986-4}
}

@article{SenthilRaani2014,
	author  = {Senthil Raani, K. S.},
	title   = {{$L^p$-integrability, dimensions of supports of
			{Fourier} transforms and applications}},
	journal = {J. Fourier Anal. Appl.},
	volume  = {20},
	number  = {4},
	pages   = {801--815},
	year    = {2014},
	doi     = {10.1007/s00041-014-9334-5}
}

@article{GuoIosevichZhangZorinKranich2025,
	author  = {Guo, Shaoming and Iosevich, Alex and Zhang, Ruixiang
		and Zorin-Kranich, Pavel},
	title   = {{$L^p$ integrability of functions with {Fourier}
			support on a smooth space curve}},
	journal = {Indiana Univ. Math. J.},
	volume  = {74},
	number  = {6},
	pages   = {1849--1853},
	year    = {2025},
	doi     = {10.1512/iumj.2025.74.60590}
}

@article{DuanPhamRyou2025,
	author  = {Duan, Shengze and Pham, Minh-Quy and Ryou, Donggeun},
	title   = {{$L^p$-integrability of functions with {Fourier}
			supports on fractal sets on the moment curve}},
	journal = {J. Funct. Anal.},
	volume  = {289},
	number  = {12},
	pages   = {111185},
	year    = {2025},
	doi     = {10.1016/j.jfa.2025.111185}
}

@article{Drury1985,
	author  = {Drury, S. W.},
	title   = {Restrictions of {Fourier} transforms to curves},
	journal = {Ann. Inst. Fourier (Grenoble)},
	volume  = {35},
	number  = {1},
	pages   = {117--123},
	year    = {1985},
	doi     = {10.5802/aif.1001}
}

@article{Hua1952,
	author  = {Hua, Loo-Keng},
	title   = {On the number of solutions of {Tarry's} problem},
	journal = {Acta Sci. Sinica},
	volume  = {1},
	pages   = {1--76},
	year    = {1952}
}

@article{ArkhipovKaratsubaChubarikov1980,
	author  = {Arkhipov, G. I. and Karatsuba, A. A. and
		Chubarikov, V. N.},
	title   = {Trigonometric integrals},
	journal = {Math. USSR-Izv.},
	volume  = {15},
	number  = {2},
	pages   = {211--239},
	year    = {1980},
	doi     = {10.1070/IM1980v015n02ABEH001212}
}

@article{Ikromov1997,
	author  = {Ikromov, I. A.},
	title   = {On the convergence exponent of trigonometric integrals},
	journal = {Tr. Mat. Inst. Steklova},
	volume  = {218},
	pages   = {179--189},
	year    = {1997}
}

@misc{Mockenhaupt1996,
	author       = {Mockenhaupt, Gerd},
	title        = {Bounds in {Lebesgue} spaces of oscillatory
		integral operators},
	howpublished = {Habilitationsschrift, Universit{\"a}t Siegen},
	address      = {Siegen},
	year         = {1996}
}

@article{BasuGuoZhangZorinKranich2026,
	author  = {Basu, Saugata and Guo, Shaoming and Zhang, Ruixiang
		and Zorin-Kranich, Pavel},
	title   = {A stationary set method for estimating oscillatory
		integrals},
	journal = {J. Eur. Math. Soc.},
	volume  = {28},
	number  = {5},
	pages   = {2067--2100},
	year    = {2026},
	doi     = {10.4171/JEMS/1506}
}

\end{document}